\documentclass{amsart}
\usepackage[utf8]{inputenc}
\usepackage{amsmath,amsthm}
\usepackage{amsfonts}
\usepackage{graphicx}
\usepackage{subfig}
\usepackage{setspace}

\newtheorem{proposition}{Proposition}
\theoremstyle{definition}
\newtheorem{definition}{Definition}

\theoremstyle{definition}

\newtheorem{theorem}{Theorem}

\newtheorem{lemma}{Lemma}
\newtheorem{corollary}{Corollary}

\newcommand{\B}[1]{\ensuremath{\mathbb{#1}}}

\title{Self-Similar Structure of $k$- and Biperiodic Fibonacci Words}
\author[Bortz Et Al]{Darby Bortz, Nicholas Cummings, Suyi Gao, Elias Jaffe, \\ Lan Mai, Benjamin Steinhurst, and Pauline Tillotson}
\address{Benjamin Steinhurst \\ McDaniel College  \\ \texttt{bsteinhurst@mcdaniel.edu}}
\date{\today}

\begin{document}

\maketitle

\tableofcontents

\begin{abstract}
Defining the biperiodic Fibonacci words as a class of words over the alphabet $\{0,1\}$, and two specializations the $k-$Fibonacci and classical Fibonacci words, we provide a self-similar decomposition of these words into overlapping words of the same type. These self-similar decompositions complement the previous literature where self-similarity was indicated but the specific structure of how the pieces interact was left undiscussed. 

\end{abstract}

\section{Introduction}

The Fibonacci sequence is a classical numerical sequence defined by a recurrence relation in the additive semigroup of \B{N} by 
$$F_1 = 0 \hspace{1in} F_2 = 1 \hspace{1in} F_n = F_{n-1} + F_{n-2}.$$
Consider instead the semigroup of words over an alphabet,  $A = \{0,1\}$. Let $A^{*} = \bigcup_{n \ge 0} A^{n}$ where the operation is concatenation. Then the sequence determined by the Fibonacci recurrence relation 
$$f_1 = 0 \hspace{1in} f_2 = 1 \hspace{1in} f_n = f_{n-1}f_{n-2}$$
is called the Fibonacci sequence of words. For example:
$$f_3 = 010 \hspace{1in} f_4 = 01001 \hspace{1in} f_5 = 01001010.$$

In \cite{MonnFrac2009} a \emph{drawing rule} was introduced to automatically create a curve in $\B{R}^{2}$ from $f_n$. The object of study there was the geometry of the drawn curves. These Fibonacci curves have scaling limits that are self-similar fractals. This paper is the first of a sequence investigating the self-similar fractal geometry of curves associated to generalizations of the Fibonacci words. In this paper we consider just the combinatorial properties of the bi-periodic Fibonacci words; the drawing rule, curves, and fractal geometry will be considered in subsequent papers. 

Fibonacci words have appeared as examples in \cite{MorseHedlund1940,Lothaire1983} and their history is discussed in \cite{Berstel1985}. Two interesting features of Fibonacci words are that 1) there is an infinite word $f$  such that every $f_n$ appears as a prefix of $f$ and 2) there is a substitution $\sigma(0) = 01$ and $\sigma(1) = 0$ such that $\sigma(f_n) = f_{n+1}$ and $\sigma(f) = f.$ The existence of a substitution like $\sigma$ is what is usually referred to as a word being \emph{self-similar} \cite{Grytczuk1996}. However, we are informed by the fractal geometry that this project is directed towards so we wish to know not just that the words are self-similar but also what the \emph{cell structure} would be. For example consider the decomposition
$$f_n = f_{n-2}f_{n-2}f_{n-4}t_{n-2}t_{n-2}.$$
(The symbol $t_n$ denotes $f_n$ with the last two digits transposed, see Definition \ref{def:tn}.) This can be read as ``a Fibonacci word is composed of five subcells four of size $n-2$ and one of size $n-4$.'' The important thing is that each of those subcells could again be subdivided using the same pattern. It will be seen below that it will be impossible to have subcells not overlap for $k-$Fibonacci words let alone for biperiodic words. In the geometry of fractals having cells overlap is a normal occurrence and so will be acceptable for us.

For the sake of this paper we define a self-similar structure on a sequence of Sturmian words. 

\begin{definition}\label{def:cellstructure}
	Let $\{f_{n}\}_{n \ge 0}$ be a sequence of Sturmian words. A \emph{cell structure} on the sequence is a decomposition of $f_n$ into possibly overlapping copies of $f_{n-l}$ or $t_{n-l}$ (see Definition \ref{def:tn}) for some fixed $l$ that is independent of $n$. The pattern of the decomposition is also assumed to be independent of $n$.
	
	We say $l$ is the period of the self-similarity structure.
\end{definition}

Theorem \ref{thm:selfsimilarity} proves that $l=2,\ 4,$ or $6$ depending on the parities of the parameters $a$ and $b$. However these values of $l$ were already known from Theorem 20 of \cite{RamirezRubiano2015}. The novel result in this paper is the cell structure not the value of the period $l$. 

Fibonacci words and their generalizations ($k-$Fibonacci and bi-periodic Fibonacci) that are introduced below are examples of Sturmian words \cite{DeLuca1997,RamirezEtAl2013,RamirezRubiano2015}. Sturmian words have a structure theorem stating that a sequence of words is Sturmian if for some base cases  and a sequence $\{q_n\}$ of positive integers then we have $s_{n+1} = s_n^{q_{n}-1}s_{n-1}$ (see \cite{DeLuca1997}.) The Fibonacci words are generated when $q_n= 2$, the $k-$nacci words when $q_n = k+1$, and the biperiodic Fibonacci words when $q_n$ is a $2-$periodic sequence. It is known \cite{CrispEtAl1993} that a Sturmian word is self-similar if and only if the sequence $\{q_n\}_{n \ge 2}$ is periodic. So the words considered in this paper are the three simplest families of self-similar Sturmian words. The techniques used in this paper have been explored by the authors for $3-$periodic sequence but that is not included as they there was no new insight gained. 

 In Section \ref{sec:biper} we define biperiodic Fibonacci words and cite many useful facts about them.  We also define an \emph{overlapping word} and prove several decomposition results using the overlapping word. Then in Section \ref{sec:decomp} we prove the main decomposition theorem about biperiodic Fibonacci words. Lastly, in Section \ref{sec:ab1} we address the edge cases where $a$ and/or $b$ equals one. 

\subsubsection*{Aknowledgements:} The authors wish to thank the McDaniel College Student-Faculty Collaboration program for funding all of the authors through multiple years. Spencer Hamblen has provided advice and insight through several years for this project. We also wish to thank Matthew O'Neill for initial observations regarding the biperiodic Fibonacci words.


\section{The Biperiodic Fibonacci Words}\label{sec:biper}

We begin with the definition of the sequence of biperiodic Fibonacci words and comment on how these become the $k-$nacci and classical Fibonacci words depending on the choice of $a$, $b$, and initial conditions. Then we will define a few related words. A few basic properites of these words will be quoted from \cite{RamirezEtAl2013} and \cite{RamirezRubiano2015}.

\begin{definition}\label{def:biperwords}
Let $a,b \ge 1$. The biperiodic Fibonacci words are defined by the recurrence relation
$$ f_{(a,b,0)} = 0 \hspace{.25in} f_{(a,b,1)} = 0^{a-1}1 \hspace{.25in} f_{(a,b,n)} = \begin{cases} f^{a}_{(a,b,n-1)}f_{a,b,n-2} & n = 2m \\ f^{b}_{(a,b,n-1)}f_{a,b,n-2} & n = 2m+1. \end{cases} $$
\end{definition}

For example when $a=2$ and $b=3$ we have: 
$$f_{(2,3,1)} = 01 \hspace{.25in} f_{(2,3,2)} = 01\ 01\ 0 \hspace{.25in} f_{(2,3,3)} = 01010\ 01010\ 01010\ 01.$$

The cases where $a$ or $b$ are equal to $1$ turn out to require more care and larger required $n$-values in most of the following results. \emph{As a default for the rest of the paper we assume that $a,b \ge 2$.} The special considerations for when $a=1$ or $b=1$ are collected in Section \ref{sec:ab1}. 

Where to begin the indexing is a matter of choice. We choose primarily to begin our indexing at $n=0$ to maintain agreement with the fractals produced in \cite{RamirezEtAl2013,RamirezRubiano2015}. While the definition of a biperiodic Fibonacci number is natural in the context of Sturmian words it also has a geometric interpretation as seen in \cite{FalconPlaza2007} which is what inspired the authors of \cite{RamirezRubiano2015} to study them. 

\begin{definition}\label{def:knacciwords}
The $k-$nacci, or $2-$Fibonacci, /words are $f_{k,n} = f_{(k,k,n)}$ for $k \ge 1$ with the same initials conditions as in Definition \ref{def:biperwords}. 

The classical Fibonacci words are $f_{n+1} = f_{(1,1,n)}$ with the initial conditions $f_{(1,1,0)} = 1$ and $f_{(1,1,1)} = 0$ instead. Changing the initial conditions just exchanges the roles of $0$ and $1$ but not the cell structures which only depend on the recurrence relation.
\end{definition}

We collect from \cite{RamirezRubiano2015} some useful features and notation for biperiodic Fibonacci words. 

\begin{proposition}\label{prop:RRprops}
Let $f_{(a,b,n)}$ be a biperiodic Fibonacci word with at least two digits. Then
\begin{enumerate}
	\item For all $n \ge 3$ let $f_{(a,b,n)} = p_{(a,b,n)}xy$. If $n$ is even then $xy=01$, if $n$ is odd then $xy = 10$.	\item For all $n \ge 3$ $p_{(a,b,n)}$ is a palindrome.
	\item For all $n \ge 3$, $f_{(a,b,n-1)}p_{(a,b,n-2)} = f_{(a,b,n-2)}p_{(a,b,n-1)}$.
\end{enumerate}
	
\end{proposition}

In light of Proposition \ref{prop:RRprops} parts 1 and 2 it is reasonable to define a ``co-biperiodic Fibonacci word.'' The following definition is quite old in the literature but is rarely given a specific name \cite{Fici2015}.

\begin{definition}\label{def:tn}
Let $t_{(a,b,n)}$ be $f_{(a,b,n)}$ with the last two symbols interchanged. That is if $f_{(a,b,n)} = p_{(a,b,n)}xy$ then $t_{(a,b,n)} = p_{(a,b,n)}yx$.
\end{definition}

The following follows immediately from Proposition \ref{prop:RRprops} part 3 and Definition \ref{def:tn}.

\begin{corollary}\label{cor:ftswap}
Let $n \ge 5$ then 
\begin{align*}
f_{(a,b,n-1)}f_{(a,b,n-2)} &= f_{a,b,n-2)}t_{(a,b,n-1)}\\
f_{(a,b,n-1)}t_{(a,b,n-2)} &= f_{(a,b,n-2)}f_{(a,b,n-1)}.
\end{align*}
\end{corollary}

Several of the formulae which appear later in this paper will have $a$ and $b$ in particular positions if $n$ is even and when $n$ is odd the exact same formula will hold with the $a$ and $b$'s exchanging roles. In order to not have to write essentially the same formula twice every time this happens we introduce the following notation.

\begin{definition}\label{def:rs}
Let $r$ be the ``most recently'' used of the $a$ or $b$ in the construction of $f_{(a,b,n)}$ and $s$ the one previous to that. More specifically:
$$ r = \begin{cases} a & n \text{ is even} \\ b & n \text{ is odd} \end{cases} \hspace{1in} s = \begin{cases} b & n \text{ is even} \\ a & n \text{ is odd} \end{cases} $$
\end{definition}

A key realization in the formation of this paper was that the natural self-similar structure was one that included overlapping cells. To that end it will be useful to express how copies of $f$ and $t$ will overlap with each other efficiently

\begin{lemma}\label{lem:fendtstart}
Let $n \ge 5$, $f_{(a,b,n)}$ ends in $f^{r-2}_{(a,b,n-1)}f_{(a,b,n-2)}$ similarly $t_{(a,b,n)}$ begins with $f^{r-2}_{(a,b,n-1)}f_{(a,b,n-2)}$.
\end{lemma}

\begin{proof}
We decompose $f_{(a,b,n)}$ and $t_{(a,b,n)}$ according to the rules allowed by Definition \ref{def:biperwords} and using the notational shortcut of Definition \ref{def:rs}:
\begin{align*}
f_{(a,b,n)} &= f^2_{(a,b,n-1)} f^{r-2}_{(a,b,n-1)} f_{(a,b,n-2)}
\end{align*}
and
\begin{align*}
t_{(a,b,n)} &= f^r _{(a,b,n-1)}t _{(a,b,n-2)}\\
 &= f^{r-2}_{(a,b,n-1)}f_{(a,b,n-1)}t_{(a,b,n-2)} \\
 &= f^{r-2}_{(a,b,n-1)}f_{(a,b,n-2)}f^{s-1}_{(a,b,n-2)}f_{(a,b,n-3)}f_{(a,b,n-1)}t_{(a,b,n-2)}.
\end{align*}
\end{proof}

\begin{definition}\label{def:In}
For $n \ge 5$ let $I_n$ be the word defined as beginning as $f_{(a,b,n)}$ and ending as $t_{(a,b,n)}$ where these two words overlap by exactly $f^{r-2}_{(a,b,n-1)}f_{(a,b,n-2)}$. Visually we can represent $I_{(a,b,n)}$ as follows where the square brackets show $f_{(a,b,n)}$ while the curly brackets show the $t_{(a,b,n)}$.
\begin{align*}
I_{(a,b,n)} &= \left[ f^2_{(a,b,n-1)} \left\{ f^{r-2}_{(a,b,n-1)}f_{(a,b,n-2)} \right] f^{s-1}_{(a,b,n-2)}f_{(a,b,n-3)}f_{(a,b,n-1)}t_{(a,b,n-2)} \right\}.
\end{align*}
\end{definition}

Lemma \ref{lem:fendtstart} is the main situation where the case of $r=1$ becomes problematic since we have $r-2=-1$ copies of $f_{(a,b,n-1)}.$ An alternative version of $I$ will be given in Section \ref{sec:ab1}. 

It will be convenient to have multiple identities for how the words $f$, $t$, and $I$ combine. We begin with a simple one.

\begin{proposition}\label{prop:I}
For all $n \ge 5$ and all $r \ge 2$ we have $I_{(a,b,n)} = f^2_{(a,b,n-1)}t _{(a,b,n)}$.	
\end{proposition}

\begin{proof}
Consider the proof of Lemma \ref{lem:fendtstart} which gives the constituent parts of $I_{(a,b,n)}$. Thus
\begin{align*}
I_{(a,b,n)} &= 	f^2_{(a,b,n-1)} f^{r-2}_{(a,b,n-1)} f_{(a,b,n-2)} f^{s-1}_{(a,b,n-2)}f_{(a,b,n-3)}f_{(a,b,n-1)}t_{(a,b,n-2)}\\
 &= f^2_{(a,b,n-1)} f^{r-2}_{(a,b,n-1)} f^s_{(a,b,n-2)}f_{(a,b,n-3)}f_{(a,b,n-1)}t_{(a,b,n-2)}\\
 &= f^2_{(a,b,n-1)} f^{r-2}_{(a,b,n-1)} f_{(a,b,n-1)}f_{(a,b,n-1)}t_{(a,b,n-2)}\\
 &= f^2_{(a,b,n-1)} f^r_{(a,b,n-1)} t_{(a,b,n-2)}\\
 &= f^2_{(a,b,n-1)} t_{(a,b,n)}.
\end{align*}
\end{proof}

The next proposition is the first where we make use of Corollary \ref{cor:ftswap}. It will be an important tool is recognizing a self-similar structure in $f_{(a,b,n)}$.

\begin{proposition}\label{prop:f^2}
Let $n \ge 6,$ then $f^2_{(a,b,n)} = f^r _{(a,b,n-1)} I _{(a,b,n-1)} t^{r-1} _{(a,b,n-1)}. $
\end{proposition}

\begin{proof}
We begin with the left hand side and begin with a decomposition via Definition \ref{def:biperwords}. We then use Corollary \ref{cor:ftswap} to repeatedly swap $f_{(a,b,n-1)}f_{(a,b,n-2)}$ with $f_{(a,b,n-2)}t_{(a,b,n-1)}$. Finally Proposition \ref{prop:I} is used to regather terms. 
\begin{align*}
f^2 _{(a,b,n)} &= f _{(a,b,n)}f _{(a,b,n)}\\
&= f^r_{(a,b,n-1)}f_{(a,b,n-2)}f^r_{(a,b,n-1)}f_{(a,b,n-2)}\\
&= 	f^r_{(a,b,n-1)}f_{(a,b,n-2)} f_{(a,b,n-2)}t^r_{(a,b,n-1)}\\
&= f^r_{(a,b,n-1)}f_{(a,b,n-2)} I_{(a,b,n-1)} t^{r-1}_{(a,b,n-1)}
\end{align*}
\end{proof}


\section{Decompositions of $f_{(a,b,n)}$}\label{sec:decomp}

We now have a collection of biperiodic Fibonacci words $f,\ t,\ I$ which are connected by considering first $f^2$. Now it will be possible to represent $f_{(a,b,n)}$ as a composition of $f_{(a,b,n-l)},\ t_{(a,b,n-l)},$ and $I_{(a,b,n-l)}.$ Where $l$ depends on the parities of $a$ and $b$ and matches the periodicity of the substitution morphism of \cite{RamirezRubiano2015}. That is, the period of the self-similar structure follows from the substitution rules but the contribution here is the \emph{arrangement and number} of the cells in the self-similar structure. 

The following sequence of lemmas and the theorem they culminate in depend on the parity of $a$, $b$, and $n$ lining up in particular ways. Specifically, the eight ways to choose parity of $a$, $b$, and $n$ break down into just three cases. 

\begin{center}
\begin{tabular}{|c|c|c|c|c|c|} \hline
Lemma			& $a$ 	&  $b$	& $n$ 	& $r$	& $s$ \\ \hline \hline
\ref{lem:reven}		& even	& even 	& even	& even	& even\\ \hline
\ref{lem:reven}		& even	& odd 	& even	& even	& odd\\ \hline
\ref{lem:reven}		& odd	& even	& odd	& even	& odd \\ \hline \hline
\ref{lem:roddseven}	& odd	& even	& even	& odd	& even \\ \hline
\ref{lem:roddseven}	& even	& odd	& odd	& odd	& even \\ \hline \hline
\ref{lem:rsodd}		& odd	& odd	& either	& odd	& odd \\ \hline
\end{tabular}
\end{center}

The lower limits on $n$ are generally chosen to accommodate applications of Proposition \ref{prop:f^2}. Similarly we assume that $a,b \ge 2$ here and leave the necessary modifications for when one or the other is $1$ to Section \ref{sec:ab1}.

\begin{lemma}\label{lem:reven}
Suppose $r$ is even and $n \ge 7$, then $f_{(a,b,n)}$ is similar to $f_{(a,b,n-2)}$ according to the formula
\begin{align*}
f_{(a,b,n)} &= \left( f^s_{(a,b,n-2)}I_{(a,b,n-2)}t_{(a,b,n-2)}^{s-1} \right)^{r/2} f_{(a,b,n-2)}.	
\end{align*}
\end{lemma}

\begin{proof}
Consider $f_{(a,b,n)}$ with $r=2$ and $n \ge 7$ using Proposition \ref{prop:f^2}:
\begin{align*}
\left( f^{s}_{(a,b,n-2)} I_{(a,b,n-2)} t^{s-1}_{(a,b,n-2)} \right)^{r/2} f_{(a,b,n-2)} &= \left( f^{2}_{(a,b,n-1)} \right)^{r/2}f_{(a,b,n-2)}.
\end{align*}
Simplifying the powers show that this is equal to $f_{(a,b,n)}.$
\end{proof}

\begin{lemma}\label{lem:rsodd}
When both $r$ and $s$ are odd and $n\ge 8$, then
\begin{align}
f_{(a,b,n)} &= \left[ \left( f^2_{(a,b,n-2)} \right)^{(s+1)/2} t_{(a,b,n-3)} \left( f^2_{(a,b,n-2)} \right)^{(s-1)/2} f_{(a,b,n-3)}\right]^{(r-1)/2}  \label{eq:rsodd} \\
 & \left( f^2_{(a,b,n-2)} \right)^{(s+1)/2} t_{(a,b,n-3)} \nonumber
\end{align}
\end{lemma}

Note that subscripts do not indicate a self-similar structure since $n-2 \neq n-3$ but recall that by Proposition \ref{prop:f^2} we know that $f^2_{(a,b,n-2)}$ splits into pieces of $f_{(a,b,n-3)}$ and $t_{(a,b,n-3)}$. Still however unless $a=b=k$ this is also not self-similar since $n$ and $n-3$ have opposite parities so $f_{(a,b,n)}$ consists of $r+1$ pieces from Definition \ref{def:biperwords} while $f_{(a,b,n-3)}$ is made up of $s+1$. So to really have a self-similar structure this lemma needs to be invoked twice, see Theorem \ref{thm:selfsimilarity}. Invoking this lemma twice does give the similarity between $f_{(a,b,n)}$ and $f_{(a,b,n-6)}$ that is indicated by the substitution morphism in \cite{RamirezRubiano2015}[Theorem 20 (2)].

\begin{proof}
Similar to the previous proof consider the right hand side of (\ref{eq:rsodd}) and simplify it to 
\begin{align*}
	\left[ f^{s+1}_{(a,b,n-2)} t_{(a,b,n-3)} f^{s-1}_{(a,b,n-2)}f_{(a,b,n-3)} \right]^{(r-1)/2}f^{s+1}_{(a,b,n-2)}t_{(a,b,n-3)}
\end{align*}
Using Corollary \ref{cor:ftswap} and Definition \ref{def:biperwords} we have the replacement:
\begin{align*}
 f^{s+1}_{(a,b,n-2)}t_{(a,b,n-3)} &= f^{s}_{(a,b,n-2)} f_{(a,b,n-3)} f_{(a,b,n-2)}
\end{align*}
In light of which we have 
\begin{align*}
	(\ref{eq:rsodd}) & = \left[ 
	f^{s}_{(a,b,n-2)} f_{(a,b,n-3)} f^{s}_{(a,b,n-2)} f_{(a,b,n-3)}
	\right]^{(r-1)/2} f^{s}_{(a,b,n-2)}f_{(a,b,n-3)}f_{(a,b,n-2)} \\
	&= \left[ f_{(a,b,n-1)} f_{(a,b,n-1)} \right]^{(r-1)/2} f_{(a,b,n-1)}f_{(a,b,n-2)}\\
	&= f^{r-1}_{(a,b,n-1)} f_{(a,b,n-1)}f_{(a,b,n-2)}\\
	&= f_{(a,b,n)}
\end{align*}
\end{proof}

\begin{lemma}\label{lem:roddseven}
When $r$ is odd and $s$ is even and $n \ge 9$ then 
\begin{align}
	f_{(a,b,n)} &= \left\{ \left[ \left( f_{(a,b,n-3)}^{2} \right)^{(r+1)/2} t_{(a,b,n-4)} \left( f^{2}_{(a,b,n-3)} \right)^{(r-1)/2} f_{(a,b,n-4)} \right]^{s/2} \right. \label{eq:roddseven}\\
	& \left( f^{2}_{(a,b,n-3)} \right)^{(r+1)/2}f_{(a,b,n-4) \nonumber }\\
	&  \left[ \left( f_{(a,b,n-3)}^{2} \right)^{(r+1)/2} t_{(a,b,n-4)} \left( f^{2}_{(a,b,n-3)} \right)^{(r-1)/2} f_{(a,b,n-4)} \right]^{(s-2)/2} \nonumber \\
	& \left.  \left( f^{2}_{(a,b,n-3)} \right)^{(r+1)/2}f_{(a,b,n-4)} \right\}^{(r-1)/2} \nonumber \\
	&  \left[ \left( f_{(a,b,n-3)}^{2} \right)^{(r+1)/2} t_{(a,b,n-4)} \left( f^{2}_{(a,b,n-3)} \right)^{(r-1)/2} f_{(a,b,n-4)} \right]^{s/2} \nonumber \\
	& \left( f^{2}_{(a,b,n-3)} \right)^{(r+1)/2}f_{(a,b,n-4)} \nonumber
\end{align}
\end{lemma}

Again, Proposition \ref{prop:f^2} lets us obtain a self-similar structure, this time at level $n-4$. This similarity agrees with \cite[Theorem 20]{RamirezRubiano2015} for even $n$ cases and also works with $a$ and $b$ exchanging roles for odd $n$ cases. There are obviously no corresponding $k-$Fiboancci formulas since this case does not appear when $a=b$. 

\begin{proof}
The strategy is the same, first we simplify the exponents in (\ref{eq:roddseven}) and use Corollary \ref{cor:ftswap} to perform the exchange
\begin{align*}
	f^{r+1}_{(a,b,n-3)}t_{(a,b,n-4)} & = f^r_{(a,b,n-3)}f_{(a,b,n-4)} f_{(a,b,n-3)}
\end{align*}

After repeated uses of this fact, the simplification of exponents, and Definition \ref{def:biperwords} we arrive at

\begin{align*}
(\ref{eq:roddseven}) &= \left\{ \left[ f^{2}_{(a,b,n-2)}\right]^{s/2} f_{(a,b,n-3)}^{r+1}f_{(a,b,n-4)} \right. \\
&= \left. \left[ f_{(a,b,n-2)}^{2} \right]^{(s-2)/2} f^{r+1}_{(a,b,n-3)} t_{(a,b,n-4)} \right\}^{(r-1)/2}\\ 
&= \left[ f^{2}_{(a,b,n-2)}\right]^{s/2} f_{(a,b,n-3)}^{r+1}f_{(a,b,n-4)}.
\end{align*}

As $s$ is even this can be further simplified to:
\begin{align*}
& \left\{ f^{s}_{(a,b,n-2)} f^{r+1}_{(a,b,n-3)} f_{(a,b,n-4)} f^{s-2}_{(a,b,n-2)} f^{r+1}_{(a,b,n-3)} t_{(a,b,n-4)}\right\}^{(r-1)/2} \\
 & \hspace{3in}f^{s}_{(a,b,n-2)} f^{r+1}_{(a,b,n-3)} f_{(a,b,n-4)}. 
\end{align*}
Then using Definition \ref{def:biperwords} and Corollary \ref{cor:ftswap} the $n-4$ terms can be combined and ultimately reduce to $f_{(a,b,n)}$. 
\end{proof}

The following theorem follows from Lemmas \ref{lem:reven} through \ref{lem:roddseven} and several uses of Proposition \ref{prop:f^2}. The lemmas contain the technical details but are not writen to make clear what the self-similar structure of a biperiodic Fibonacci word is. The following theorem is a restatement of the lemmas making it clear that $f_{(a,b,n)}$ is an overlapping concatenation of similar words. 

\begin{theorem}\label{thm:selfsimilarity}
The biperiodic Fibonacci words $f_{(a,b,n)}$ are self-similar in the sense of being composed of overlapping copies of $f_{(a,b,n-l)}$ with appropriate values of $l$ depending on the parities of $a$ and $b$. If $a$ and $b$ are both odd the displayed self-similar structure must be applied twice to be truly self-similar. The self-similar structures are displayed in Table \ref{tab:SSS}. The minimal $n$ that these structures hold for are the same as in Lemmas \ref{lem:reven} through \ref{lem:roddseven}.
\end{theorem}

\begin{table}[tp] 
\begin{center}
\begin{tabular}{|rp{4in}|}\hline
& \\
$f_{(a,b,n)} =$	& $\left( f^{s}_{(a,b,n-2)} I_{(a,b,n-2)} t^{s-1}_{(a,b,n-2)} \right)^{r/2}f_{(a,b,n-2)}$ \\ 
 & \\ \hline
 & \\
$f_{(a,b,n)} =$	& $ \left[\left( f_{(a,b,n-3)}^{r} I_{(a,b,n-3)} t_{(a,b,n-3)}^{r-1}\right)^{(s+1)/2} t_{(a,b,n-3)} \right. $\\
& $ \left.\left( f_{(a,b,n-3)}^{r} I_{(a,b,n-3)} t_{(a,b,n-3)}^{r-1}\right)^{(s-1)/2} f_{(a,b,n-3)} \right]^{(r-1)/2}  $\\ 
& $ \left( f_{(a,b,n-3)}^{r} I_{(a,b,n-3)} t_{(a,b,n-3)}^{r-1}\right)^{(s+1)/2} t_{(a,b,n-3)} $ \\
& \\ \hline
& \\
$f_{(a,b,n)} =$ & $\left\{ \left[ \left(  f^{s}_{(a,b,n-4)} I_{(a,b,n-4)} t_{(a,b,n-4)}^{s-1} \right)^{(r+1)/2} t_{(a,b,n-4)} \right. \right. $ \\
& $ \left. \left(  f^{s}_{(a,b,n-4)} I_{(a,b,n-4)} t_{(a,b,n-4)}^{s-1} \right)^{(r-1)/2} f_{(a,b,n-4)}\right]^{s/2} $ \\ 
& $ \left(  f^{s}_{(a,b,n-4)} I_{(a,b,n-4)} t_{(a,b,n-4)}^{s-1} \right)^{(r+1)/2} f_{(a,b,n-4)}$ \\
& $ \left[ \left(  f^{s}_{(a,b,n-4)} I_{(a,b,n-4)} t_{(a,b,n-4)}^{s-1} \right)^{(r+1)/2} t_{(a,b,n-4)} \right. $ \\
& $ \left. \left(  f^{s}_{(a,b,n-4)} I_{(a,b,n-4)} t_{(a,b,n-4)}^{s-1} \right)^{(r-1)/2} f_{(a,b,n-4)}\right]^{(s-2)/2} $ \\ 
& $ \left. \left(  f^{s}_{(a,b,n-4)} I_{(a,b,n-4)} t_{(a,b,n-4)}^{s-1} \right)^{(r+1)/2} t_{(a,b,n-4)} \right\}^{(r-1)/2}$ \\
& $ \left[ \left(  f^{s}_{(a,b,n-4)} I_{(a,b,n-4)} t_{(a,b,n-4)}^{s-1} \right)^{(r+1)/2} t_{(a,b,n-4)} \right. $ \\
& $ \left. \left(  f^{s}_{(a,b,n-4)} I_{(a,b,n-4)} t_{(a,b,n-4)}^{s-1} \right)^{(r-1)/2} f_{(a,b,n-4)}\right]^{s/2} $ \\ 
& $  \left(  f^{s}_{(a,b,n-4)} I_{(a,b,n-4)} t_{(a,b,n-4)}^{s-1} \right)^{(r+1)/2} f_{(a,b,n-4)} $ \\
& \\ \hline
\end{tabular}
\end{center}
\caption{The self-similar structures of for $f_{(a,b,n)}$. The first is for when $r$ is even. The second for when $r$ and $s$ are both odd, however in this case the self-similar structure requires that this decomposition be performed twice. The third is for when $r$ is odd but $s$ is even. }
\label{tab:SSS}
\end{table}%

It is evident then that when both $a$ and $b$ are even $l=2$. When one of $a$ or $b$ is odd  $l = 4$. When both are odd $l=6.$ These values of $l$ are the same as those in Theorem 20 of \cite{RamirezRubiano2015} but that the specific cell structure is novel. 

Recall that the original motivation for seeing the self-similar structure of $f_{(a,b,n)}$ was to understand the geometry of fractals which are produced from the words by means of a scaling limit construction. In forthcoming works by the same authors the self-similar decompositions of $f_{(a,b,n)}$ shown in Table \ref{tab:SSS} will be used to identify iterated functions systems that produce the same fractals. The benefit of this will be to ease the analysis of the geometry of such fractals as the theory of iterated functions systems is extensive. In particular by providing an iterated function system it becomes possible to state when the fractal associated to a biperiodic Fibonacci word is homeomorphic to a line and when it is not. 


\section{When $a$ or $b$ is $1$}\label{sec:ab1}
When $a$ or $b$ is equal to $1$ some modifications of the foregoing need to be made. The necessity of the modifications is clearly demonstrated by attempting to write $t_{(1,b,1)}$. Since $f_{(1,b,1)} = 1$ is it difficult to know how to interpret $t_{(1,b,1)}$ by transposing the last \emph{two} symbols in $f_{(1,b,1)}$. Any result in this paper that involved the use of $t_{(a,b,n)}$ will need the minimal value of $n$ assumed increased by $1$ to accommodate $t_{(1,b,1)}$ not being defined. For the sake of brevity we do not provide any of the modified proofs for the $a=1$ or $b=1$ cases but simply indicate what the new hypotheses must be. The proofs use the same methods and techniques and the details can quickly be provided by any motivated reader. 

The other modification that needs to be made is that Definition \ref{def:In} does not make sense when $r=1$.  So instead Proposition \ref{prop:I} can be proven using the following alternative definition of $I_{(a,b,n)}$ when $r=1$. 

\begin{definition}\label{def:Ivariant}
For $n \ge 5$ and $r=1$ let $I_{(a,b,n)}$ consist of the words $f_{(a,b,n)}$, $f_{(a,b,n)}$, and $t_{(a,b,n)}$ in order where each adjacent pair of words overlaps by a copy of $f_{(a,b,n-2)}$. 
\end{definition}

With this alternate definition of $I$ Theorem \ref{thm:selfsimilarity} holds with the same conclusion in the case where $r=1$.

\bibliography{FibFracCollected}{}

\begin{thebibliography}{10}

\bibitem{Berstel1985}
Jean Berstel.
\newblock Fibonacci words, a survey.
\newblock In J.~Avenhaus and A.~Salomaa, editors, {\em The Book of {L}}, pages
  11--25. Springer-Verlag New York Inc., New York, 1985.

\bibitem{CrispEtAl1993}
D.~Crisp, W.~Moran, A.~Pollington, and P.~Shiue.
\newblock Substitution invariant cutting sequences.
\newblock {\em Journal de Th\'eorie des Nombres de Bordeaux}, 5(1):123--137,
  1993.

\bibitem{DeLuca1997}
Aldo {de Luca}.
\newblock Sturmian words: structure, combinatorics, and their arithmetics.
\newblock {\em Theoretical Computer Science}, 183(1):45--82, 1997.
\newblock Formal Language Theory.

\bibitem{FalconPlaza2007}
Sergio Falc{\'o}n and {\'A}ngel Plaza.
\newblock On the fibonacci k-numbers.
\newblock {\em Chaos, Solitons \& Fractals}, 32(5):1615--1624, 2007.

\bibitem{Fici2015}
Gabriele Fici.
\newblock Factorizations of the fibonacci infinite word.
\newblock {\em Journal of Integer Sequences}, 18, 08 2015.

\bibitem{Grytczuk1996}
Jaros{\l}aw Grytczuk.
\newblock Infinite self-similar words.
\newblock {\em Discrete Mathematics}, 161(1):133--141, 1996.

\bibitem{Lothaire1983}
M.~Lothaire.
\newblock {\em Combinatorics on words}, volume~17 of {\em Encyclopedia of
  Mathematics and its Applications}.
\newblock Addison-Wesley Publishing Co., Reading, Mass., 1983.
\newblock A collective work by Dominique Perrin, Jean Berstel, Christian
  Choffrut, Robert Cori, Dominique Foata, Jean Eric Pin, Guiseppe Pirillo,
  Christophe Reutenauer, Marcel-P. Sch\"{u}tzenberger, Jacques Sakarovitch and
  Imre Simon, With a foreword by Roger Lyndon, Edited and with a preface by
  Perrin.

\bibitem{MonnFrac2009}
Alexis Monnerot-Dumaine.
\newblock The {F}ibonacci word fractal, 2009.
\newblock Accessed 2015-06-02.

\bibitem{MorseHedlund1940}
Marston Morse and Gustav~A. Hedlund.
\newblock Symbolic dynamics ii. sturmian trajectories.
\newblock {\em American Journal of Mathematics}, 62(1):1--42, 1940.

\bibitem{RamirezEtAl2013}
Jos\'e~L. {Ram\'irez} and Gustavo~N. {Rubiano}.
\newblock {On the $k$-Fibonacci words.}
\newblock {\em {Acta Univ. Sapientiae, Inform.}}, 5(2):212--226, 2013.

\bibitem{RamirezRubiano2015}
Jos\'e~L. Ram\'irez and Gustavo~N. Rubiano.
\newblock Biperiodic fibonacci word and its fractal curve.
\newblock {\em Acta Polytechnica}, 55(1):50--58, 2015.

\end{thebibliography}
\bibliographystyle{plain}

\end{document}